\documentclass[12pt]{amsart}
\usepackage[hmargin={1in, 1in}, vmargin={1in, 1in}]{geometry}
\usepackage{amsmath}
\usepackage{amsfonts}
\usepackage{graphicx}
\usepackage{setspace}

\usepackage{fancyhdr}
\title{Repeated binomial coefficients and high-degree curves}
\author{Hugo Jenkins}
\begin{document}
\maketitle
\begin{abstract}
We consider the problem of characterizing solutions in $(x, y)$ to the equation ${x \choose y}={{x-a} \choose {y+b}}$ in terms of $a$ and $b$. We obtain one simple result which allows the determination of a ratio in terms of $a$ and $b$ which the ratio $\frac{x}{y}$ must approximate. We then add to the understanding of the infinite family of repeated coefficients discovered by D. Singmaster, by using fundamental results from Diophantine geometry to prove that in the case $a \neq b$, solutions to ${x \choose y}={{x-a} \choose {y+b}}$ are finite. Finally, we make some observations about the potential utility of equations of the form ${x \choose y}={{x-a} \choose {y+b}}$ in proving Singmaster's conjecture, which is the main unsolved problem in the area of repeated binomial coefficient study. We remark that this approach to the conjecture is markedly different from previous approaches, which have only established logarithmic bounds on a function which counts the number of representations of $t$ as a binomial coefficient.\par
\end{abstract}
\section{Introduction}
\noindent The sequence of binomial coefficients is one of the most well-studied, frequently-used, and generally significant sequences in all of mathematics. It is interesting, therefore, that the analysis of repeated binomial coefficients (coefficients which occur more often than the trivial two times which every number occurs) has only received sustained attention in the past 50 years. Clearly, many numbers occur three and four times; these are what fill up the inside of Pascal's triangle. However, the only other high multiplicities known to occur\textemdash 6 and 8\textemdash are rare, and the patterns in which they appear are not well understood. \par
However, some scattered progress has been made.
It has been shown by Abbott \emph{et al.} [1] that the function $N(t)$ which counts the number of ways of representing $t$ as a binomial coefficient has average order and normal order 2.
All solutions to ${x \choose 2}={y \choose 3}$ were found by Avanesov [2, referenced in 10], and Kiss [10] established the more general result of finiteness of solutions to ${x \choose 2}={y \choose p}$, for $p$ a prime. ${x \choose 2}={y \choose 4}$ and ${x \choose 3}={y \choose 4}$ have also been completely solved by de Weger [17]. More recently, Bugeaud \emph{et al.} [4] have found all solutions to ${x \choose 2}={y \choose 5}$ using an improvement of the Mordell-Weil sieve, which is applicable to finding integral points on all hyperelliptic curves.\par
Perhaps the most striking result was found by Lind [11], who showed that if $n=F_{2i+2}F_{2i+3}-1$ and $k=F_{2i}F_{2i+3}-1$ (where $F_i$ is the $i$-th Fibonacci number), then ${{n+1} \choose {k+1}}={n \choose {k+2}}$. David Singmaster [16] also provided a proof of this, and noted that his result provides an infinite family of numbers with multiplicity at least 6. The first member of this family\textemdash3003\textemdash is also the only known number with multiplicity 8.
Singmaster [15] also made the following dramatic conjecture, the study of which has been an important feature of subsequent work on repeated binomial coefficients: If $N(t)$ denotes the number of times $t$ occurs in Pascal's triangle, $N(t)=O(1)$.\par
There has been no direct attempt at proving the existence of such a finite upper bound on the number of ways $t$ may be represented as a binomial coefficient. Bounds on $N(t)$ in terms of $t$ were obtained first by Singmaster [15], then by Abbott \emph{et al.} [1], and then by Kane [8]. Currently the best unconditional bound is $N(t)=O(\frac{(\log{t})(\log{\log{\log{t}}})}{(\log{\log{t}})^3})$, obtained by Kane [9] via an argument relating integer solutions of ${x \choose y}=m$ to derivatives of a function implicitly defined in terms of the $\Gamma$-function. Conditional on Cram\'{e}r's conjecture about small gaps between prime numbers, Abbott \emph{et al.} [1] obtained $N(t)=O((\log{t})^{\frac{2}{3}})$.\par
The purpose here will be to provide information about generalizations of the equation solved by Lind [11] and Singmaster [16]: ${{n+1} \choose {k+1}}={n \choose {k+2}}$. This is equivalent to ${n \choose k}={{n-1} \choose {k+1}}$. However, to our knowledge, no studies of equations of the general form ${n \choose k}={{n-a} \choose {k+b}}$ have been made. We will present two independent results about the solutions to such equations; one describes where solutions may occur, and the other asserts the finiteness of solutions (in most cases). We will also provide a rationale for why considering such equations may provide a powerful framework for proving the Singmaster conjecture itself.
\section{Results}
\noindent First, we make the following basic proposition about the location of repeats.
\newtheorem*{prop}{Proposition} 
\begin{prop}
Let $x$, $y$, $a$, and $b$ be natural numbers such that $a<y$; let $\zeta$ be the positive number defined by $\zeta^{a+b}-(\zeta+1)^a=0$.
If ${x \choose y}={{x-a} \choose {y+b}}$, we have $\frac{x-a-y-b+1}{y+b}<\zeta<\frac{x-y}{y-a+1}$.
\end{prop}
\begin{proof}
First, note the elementary fact that any entry in Pascal's triangle may be written as the sum of the two entries above it: ${n \choose k}={{n-1} \choose k}+{{n-1} \choose {k-1}}$.
This process may be iterated to obtain a representation of ${n \choose k}$ as a sum of binomial coefficients of any row number less than $n$. For example, with two and three iterations, we obtain, respectively,
\begin{equation*}
{n \choose k}={{n-2} \choose {k-2}}+2{{n-2} \choose {k-1}}+{{n-2} \choose k},
\end{equation*}
\begin{equation*}
{n \choose k}={{n-3} \choose {k-3}}+3{{n-3} \choose {k-2}}+3{{n-3} \choose {k-1}}+{{n-3} \choose k}.
\end{equation*}
That the coefficients appearing in the $r$-th such iterate correspond to the binomial coefficients of row $r$ follows from the observation that when generating the coefficient for the next term with $k-s$, one adds together the coefficients of the current terms with $k-s$ and $k-s+1$. This is exactly the process which ordinarily generates the binomial coefficients in Pascal's triangle.\par
We therefore may always write the equation ${x \choose y}={{x-a} \choose {y+b}}$ as
\begin{equation}
{{x-a} \choose {y-a}}+a{{x-a} \choose {y-a+1}}+{a \choose 2}{{x-a} \choose {y-a+2}} \dots + {{x-a} \choose {y}}={{x-a} \choose {y+b}}
\end{equation}
Next, we note that if $k<\frac{n}{2}$ and $n$ and $k$ are large, the ratios of successive binomial coefficients ${n \choose {k+1}}:{n \choose k}$, ${n \choose {k+2}}:{n \choose {k+1}}$, and so on, are strictly decreasing, and are close to being constant. Specifically, ${n \choose {k+1}}/{n \choose k}=\frac{n-k}{k+1}$. Suppose we call the ratios ${{x-a} \choose {y-a+1}}/{{x-a} \choose {y-a}}$, ${{x-a} \choose {y-a+2}}/{{x-a} \choose {y-a+1}}$, ${{x-a} \choose {y-a+3}}/{{x-a} \choose {y-a+2}}$ \dots $r_1$, $r_2$, $r_3$ and so on. Then we may rewrite (1) as
\begin{equation}
1+ar_1+{a \choose 2}r_1r_2+{a \choose 3}r_1r_2r_3 \dots + r_1r_2r_3\dots r_{a-1}=r_1r_2r_3 \dots r_{a+b}
\end{equation}
When $x$ and $y$ are very large in comparison to $a$ and $b$, all the $r_i$ are approximately the same (because of the expression for the ratio of successive binomial coefficients), and hence by the binomial theorem are all approximately the (positive) solution of $(\zeta + 1)^a=\zeta^{a+b}$.\par
(2) would be true if $\zeta=r_1=r_2=r_3\dots=r_{a+b}$. However, because of the strict decrease mentioned above, we have $r_1>r_2>r_3\dots>r_{a+b}$. Suppose, then, that $r_1<\zeta$. Then all the $r_i$ are, and the right side of (2) has experienced a proportional decrease from $\zeta^{a+b}$ which is the product of all the proportional decreases in the individual $r_i$. However, the left side cannot have experienced so great a decrease from $(\zeta+1)^a$, since no term has decreased proportionally more than the right side, and there is one term (the constant, 1) which has not decreased at all. Thus the equation (2) can no longer be true. We apply the same argument to find that $r_{a+b}$ cannot be greater than $\zeta$. \par Writing out $r_1={{x-a} \choose {y-a+1}}/{{x-a} \choose {y-a}}=\frac{x-y}{y-a+1}$ and $r_{a+b}={{x-a} \choose {y+b}}/{{x-a} \choose {y+b-1}}=\frac{x-a-y-b+1}{y+b}$ yields the inequality in the Proposition. We must impose the condition $a<y$, because in reformulating equation (1) as equation (2), we have assumed that we may divide through by the leftmost term ${{x-a} \choose {y-a}}$, which is nonzero iff $a<y$.
\end{proof}
\newtheorem*{thm}{Theorem}
\begin{thm}
If $b\neq a$, the equation ${x \choose y}={{x-a} \choose {y+b}}$ has finitely many solutions in natural numbers $x$, $y$.
\end{thm}
Any such equation can be written as the equation of an algebraic curve $\mathcal{C}$
\begin{equation}
\prod_{r=0}^{a+b-1}(x-y-r)-\prod_{p=0}^{a-1}(x-p)\prod_{q=1}^b(y+q)=0
\end{equation}
in $x$ and $y$. For example, the equation ${x \choose y}={{x-1} \choose {y+1}}$, which Singmaster [16] solved, corresponds to the curve $(x-y)(x-y-1)-x(y+1)=0$. This means that the proof of the theorem is reduced to the well-studied problem of determining whether an algebraic curve has an infinity of lattice points. For an individual curve, the standard approach to such a problem is to determine that the curve is irreducible and has genus greater than 0. If so, then by Siegel's theorem [7, p. 353] the set of lattice points on the curve is finite. However, determining the genus alone requires an analysis of singularities [7, p. 72]. This, in turn, amounts to solving the system where the three partial derivatives of the homogeneous version of the curve are simultaneously equated to zero [5, p. 19]; a task which appears to be unattackable for general curves of this complexity.\par
We therefore will not attempt to prove that Siegel's theorem is directly applicable to these curves, but instead will make use of the following criterion given by Nagell (originally due to Maillet) [12, p. 264].\par
\begin{center}
\begin{minipage}{.8\textwidth}
\singlespacing A unicursal [genus 0] curve passes through an infinity of lattice points if and only if there exists a parametric representation of the form
\begin{equation*}
x=\frac{f(t)}{(h(t))^n}, \hspace{17 mm} y=\frac{g(t)}{(h(t))^n}
\end{equation*}
where $n$ is a natural number, and where $f(t)$, $g(t)$, and $h(t)$ are integral polynomials in $t$ satisfying one of the following conditions:\\
1. Either $h(t)=at+b$ with $\gcd{a, b}=1$ or $h(t)=1$; $f(t)$ and $g(t)$ are both of degree $n$;\\
2. $h(t)=at^2+bt+c$ is irreducible, and $a>0$, $b^2-4ac>0$; $f(t)$ and $g(t)$ are both of degree $2n$; the form $au^2+buv+cv^2$ can represent for integral values of $u$ and $v$ a certain integer $k\neq0$ such that $k^n$ divides all the coefficients of both $f(t)$ and $g(t)$.
\end{minipage}
\end{center} \hspace{0mm}\\
\indent Nagell goes on to state that shortly after Maillet gave this criterion, Siegel proved that it applies to all curves (i.e. not just unicursal ones).\par To apply this criterion to our curves, it will be necessary to consider the limiting behavior as $y\rightarrow \infty$. We may assume the curves have points of arbitrarily large $y$; if they did not, they could not have arbitrarily large $x$ either (since clearly $\lim_{y \to \infty} \frac{x}{y} \neq \infty$) and so could only pass through finitely many points with natural $x$, $y$, which are the only points which matter for the theorem about binomial coefficients.\par Qualitatively, it is clear that $\lim_{y \to \infty} \frac{x}{y}$ must be such that the highest total degree terms in the equation of $\mathcal{C}$ almost cancel each other out as $y \rightarrow \infty$. Formally if $T_n$ is the $n$-th term with total degree $a+b$, 
\begin{equation}
\lim_{y \to \infty} \frac{\sum T_n}{y^{a+b}}=0, 
\end{equation}
because otherwise, for large $y$, the value of $\sum T_n$ would be the only $O(y^{a+b})$ term in $\mathcal{C}$. There would be other terms $O(y^{a+b-1})$, $O(y^{a+b-2})$, and so on, but even a direct sum of these is not $O(y^{a+b})$, and they clearly are not all summed together. \par
What is the sum of the highest total degree terms in (3)? The degree $a+b$ terms from the first product,   
$\prod_{r=0}^{a+b-1}(x-y-r)$, are simply the terms of $(x-y)^{a+b}$. There is only one degree $a+b$ term in the second product; it is $x^ay^b$. Equation (4) then becomes
\begin{equation*}
\lim_{y \to \infty} \frac{(x-y)^{a+b}-x^ay^b}{y^{a+b}}=0;
\end{equation*}
\begin{equation*}
\lim_{y \to \infty} \left( \frac{x^{a+b}}{y^{a+b}}-{{a+b} \choose 1}\frac{x^{a+b-1}}{y^{a+b-1}}+{{a+b} \choose 2}\frac{x^{a+b-2}}{y^{a+b-2}}\dots - \frac{x^{a}}{y^{a}} \right)=0.
\end{equation*}
\par
Therefore, if we take $c=\lim_{y \to \infty} \frac{x}{y}$, we must have $(c-1)^{a+b} -c^a=0$. \par
Now, consider the form of $\lim_{y \to \infty} \frac{x}{y}$ if there exists a parametric representation as described in Nagell's [12, p. 264] criterion. If $h(t)=1$, $y=g(t)$, and $y$ goes to $\infty$ as $t$ does. This means that $\lim_{y \to \infty} \frac{x}{y}=\lim_{t \to \infty} \frac{f(t)}{g(t)}$, which has a constant, rational value, because $f(x)$ and $g(x)$ are integral polynomials of the same degree in $t$. This cannot be the case, because a simple application of the rational root test shows that $c$ is irrational.\par
By Nagell's criterion, it must then be that $x$ and $y$ are given by rational functions of $t$ with numerator and denominator polynomials of the same degree. Then $y$ can only approach infinity when $t$ approaches one of the roots of the denominator, $h(t)$, i. e. when $t$ approaches either a rational or quadratic irrational number $\alpha$. We thus have that $c=\lim_{y \to \infty} \frac{x}{y}=\lim_{t \to \alpha}\frac{f(t)}{g(t)}$. $\lim_{t \to \alpha}\frac{f(t)}{g(t)}$ is clearly the quotient of two quadratic irrationals; because $f$ and $g$ have the same input $\alpha$, the number under the radical in both quadratic irrationals is the same. This means that the quotient is itself a quadratic irrational, by rationalization of denominators. We will now show that if $a\neq b$, $c$ cannot be a quadratic irrational, and hence no parametrization of the type described can exist. For convenience, we will work with the equation $c^{a+b}-(c+1)^a=0$, instead of the original $(c-1)^{a+b}-c^a=0$; the former is shifted 1 unit to the left, and obviously has quadratic zeros iff the original does.
\newtheorem*{lem}{Lemma}
\begin{lem}
If $n$ and $r$ are such that $n>r$ and $\frac{n}{r} \neq 2$, the polynomial $P(x)=x^n-(x+1)^r$ has no real roots of degree 2.
\end{lem}
\begin{proof}
We will attack this by showing that it is impossible for $P$ to have a quadratic factor with positive discriminant. We begin by noting that since $P(x)$ is primitive, by Gauss' lemma [13, p. 49], it suffices to consider quadratic factors with integral coefficients. Since the first and last terms of $P(x)$ have magnitude 1, any such factor must be of the form $\pm x^2+bx\pm1$ or $\pm x^2+bx\mp1$, with $b\in \mathbb{Z}$.\par
Also note that by Descartes' sign test, $P(x)$ has exactly 1 positive root. Make the substitution $x\mapsto x-1$, generating the new polynomial $G(x)=(x-1)^n-x^r$, which has been shifted 1 unit to the right. Substituting $-x$ for $x$ in $G$, we see that regardless of the parity of $n$ and $r$, there are no sign changes. Thus, $G$ has no negative roots. We conclude that $P(x)$ has no negative roots smaller than $-1$.\par
We have $P(0)=-1$, and $P(-1)=\pm1$, depending on whether $n$ is even or odd. This means that any quadratic factor $Q$ must take the values $\pm1$ at $x=0$ and $x=-1$. If $Q(-1)=1$, we have the following four cases:
 \begin{equation*}
Q(-1)=1=(-1)^2+b\times(-1)+1
\end{equation*}
\begin{equation*}
Q(-1)=1=-(-1)^2+b\times(-1)-1
\end{equation*}
\begin{equation*}
Q(-1)=1=(-1)^2+b\times(-1)-1
\end{equation*}
\begin{equation*}
Q(-1)=1=-(-1)^2+b\times(-1)+1
\end{equation*}
which yield, respectively, $b=1$, $b=-3$, $b=-1$, and $b=-1$. The values obtained by the same process for $Q(-1)=-1$ are, in order, $b=3$, $b=-1$, $b=1$, and $b=1$. The four cases where the first and last terms of $Q$ have the same sign generate only two polynomials with distinct roots, as do the cases where they have opposite signs. The complete list of quadratic factors of $x^n-(x+1)^r$ to be considered is thus
\begin{equation*}
x^2+x+1 \hspace{10 mm} x^2+3x+1 \hspace{10 mm} x^2-x-1 \hspace{10 mm} x^2+x-1
\end{equation*}
The first has no real roots. The second has a root $-\frac{3}{2}-\frac{\sqrt{5}}{2}$, which is less than $-1$; therefore it cannot be a factor, by the sign test performed earlier. Neither can the last, because of the root $-\frac{1}{2}-\frac{\sqrt{5}}{2}$.\par
The only possible quadratic factor is thus $x^2-x-1$. We observe that if $\frac{n}{r}=2$, this \emph{is} a factor, as can be seen from writing $x^{2r}-(x+1)^r=0$, adding one term to the other side, and taking roots. It is then easily seen that no other polynomials $x^n-(x+1)^r$ can share this factor, for 
if they did, the difference $x^{2r}-(x+1)^r-(x^n-(x+1)^r)=x^{2r}-x^n$ must also share the factor, which it clearly does not. This proves the lemma, and by extension the theorem.
\end{proof}

\section{Analysis of Methods and Intuitive Explanation}
\noindent The results obtained here describe instances where the numbers in a particular ``configuration" in the triangle are the same. The most basic instance of this, the ``configuration"
\begin{center} \includegraphics{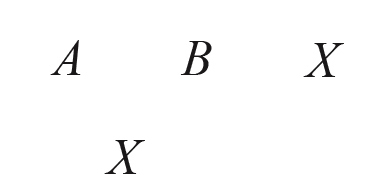} \end{center}
was shown by Singmaster [16] and Lind [11] to occur infinitely many times; in fact, precisely when $n$ and $k$ are certain expressions given by Fibonacci numbers. We have shown that configurations such as
\begin{center} \includegraphics{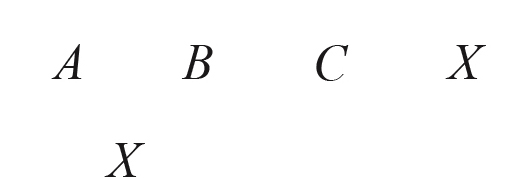} \hspace{10mm} and \hspace{7mm} \includegraphics{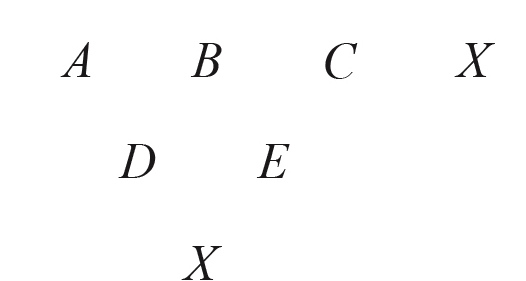} \end{center}
can occur only finitely many times, if at all. But we have \emph{not} shown, for example, that
\begin{center} \includegraphics{configurationsing} \hspace{10mm} and \hspace{7mm}  \includegraphics{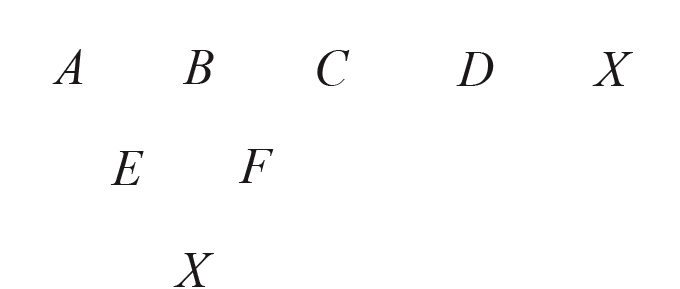} \end{center}
occur finitely many times, because in those cases the difference in $k$-values is equal to the difference in $n$-values, and the associated polynomial $x^{2r}-(x+1)^r$ has the quadratic irrational roots $\varphi$ and $-\frac{1}{\varphi}$. However, the assertion that solutions are finite is still nothing more than asserting that a certain subclass of the curves studied are irreducible and have fewer than the maximum allowable number of singularities ($\frac{(d-1)(d-2)}{2}$ [7, p. 72], barring the possibility of non-ordinary singularities), something which seems very likely.\par
We will now analyze one of the higher-degree analogues to the curve $(x-y)(x-y-1)-x(y+1)=0$ in order to illustrate the validity of this idea. As we will see, the reason this is difficult in general is because of the necessity of computing a Gr\"{o}bner basis to determine that the polynomial and its two partial derivatives share no common zeros. \par 
In the case of the next curve with possibly infinite lattice points (the curve with $a=2$, $b=2$; defined by $F(x,y)=(x-y)(x-y-1)(x-y-2)(x-y-3)-x(x-1)(y+1)(y+2)=0$), we may mechanically compute the Gr\"{o}bner basis [13, pp. 221, 237] for the system $F=\frac{\partial F}{\partial x}=\frac{\partial F}{\partial y}=0$ [5, p. 19] to see that there are no affine singularities. If we then homogenize coordinates, and consider the system $\frac{\partial F}{\partial x}=\frac{\partial F}{\partial y}=\frac{\partial F}{\partial z}=0$ [5, p. 19] at $z=0$, we see that the only solution must be $[x : y : z]=[0 : 0 : 0]$, which is not a valid point [7, p. 12]. This is because $\frac{\partial F}{\partial x}$, $\frac{\partial F}{\partial y}$, and $\frac{\partial F}{\partial z}$ are all homogeneous polynomials in $x$ and $y$ when $z=0$. By the same argument we have used to determine $\lim_{y \to \infty} \frac{x}{y}$, any homogeneous polynomial in two variables represents the union of some (possibly complex) lines through the origin. None of these lines are the same, and thus the only solution to this system is $(0, 0)$. We conclude that $F$ has no singularities.\par 
We may therefore apply the genus-degree formula without subtraction of additional terms: $g=\frac{(d-1)(d-2)}{2}=\frac{3\times2}{2}=3$ [7, p. 72]. That $F$ is irreducible follows immediately from its being nonsingular, for by B\'{e}zout's theorem [7, p. 84], any hypothetical components of $F$ must intersect somewhere in the complex projective plane, and thus create a singularity in $F$. By Siegel's theorem [7, p. 353], therefore, the set of lattice points is finite. \par
In the Proposition, we have shown that if a certain configuration occurs entirely within the triangle, the smooth function giving the ratio of one binomial coefficient to the preceding one must take a value $\zeta$ (1 plus the root of the associated polynomial) between the ``beginning" of the configuration and the ``end".\par
This is essentially a precise way of stating that all the occurrences of a particular configuration have approximately the same ratio $\frac{n}{k}$. The restriction $y>a$, which was algebraically necessary to avoid dividing by zero, corresponds to requiring that the configuration is not ``cut off" by the edge of the triangle.
All currently known nontrivial repetitions (excluding Singmaster's [16]) occur so close to the side of the triangle that the Proposition does not apply; however, it is still satisfied, because the ratios on the edge are very large and change very rapidly. It is easily seen that there cannot be more than $a$ of the ``cutoff" cases, because for each $y$, there is clearly at most one $x$ where ${x \choose y}={{x-a} \choose{y+b}}$. \par
In the case of Singmaster's infinite family, the ratio $\varphi$ is always less than ${{n-1} \choose {k+1}}/{{n-1} \choose k}$, and greater than ${{n-1} \choose k}/{{n-1} \choose {k-1}}$. This can also be seen as a direct result of working out ratios of the given expressions involving Fibonacci numbers ($n=F_{2i+2}F_{2i+3}-1$ and $k=F_{2i}F_{2i+3}-1$). It works out that the ratios ${{n-1} \choose {k+1}}/{{n-1} \choose k}$ and ${{n-1} \choose k}/{{n-1} \choose {k-1}}$ are ratios of successive pairs of Fibonacci numbers\textemdash successive continued fraction convergents to $\varphi$. In this sense, the coefficient repetition occurs at all the ``best possible" approximations to $\varphi$. It is tempting to think that this is somehow necessary for repetitions to occur, and then to try and disprove the existence of \emph{any} other repeats ``deep" in the triangle by proving that convergents to the other, non-quadratic ratios cannot occur sequentially in this way. This seems plausible because, even without invoking the more rapid continued fraction convergence properties of higher degree algebraic numbers, we have that the maximum difference between consecutive continued fraction convergents with first denominator $q$ is less than $\frac{3}{2q^2}$ [6, p. 152], which is very often less than the difference $\frac{n+1}{(k+1)(k+2)}$ between consecutive coefficient ratios. However, there is no such obvious argument for the ``necessity" of continued fraction convergence.
\section{Possible Extensions}
\noindent The most obvious extension of our work would be to show that the curve defined by $(x-y)(x-y-1)-x(y+1)=0$ is the only one of this family of curves which passes through infinitely many lattice points, i. e. to extend the Theorem to the case when $a=b$ and $a\neq 1$. To do that, an entirely different argument to the one used in this paper would be necessary, since we have relied on the fact that the limiting ratio of $x$ to $y$ in most cases is not quadratic. If $a=b$, it is quadratic, and there is no apparent way to prove that Nagell's [12, p. 264] criterion cannot be satisfied. It is possible that the symmetry of the polynomial defining the curve when $a=b$ allows a simple algebraic manipulation of the system where it and its two partial derivatives are set equal to 0, such that an inconsistency is derived. As we have seen from the previous consideration of $(x-y)(x-y-1)(x-y-2)(x-y-3)-x(x-1)(y+1)(y+2)=0$, we may work with this system, instead of $\frac{\partial F}{\partial x}=\frac{\partial F}{\partial y}=\frac{\partial F}{\partial z}=0$, because the absence of singular points at infinity follows easily for all these curves.\par 
Another idea would be to try and use the fact that $(x-y)(x-y-1)-x(y+1)=0$ is nonsingular in order to establish the non-singularity of the higher-degree curves by induction. If we designate the first large product in the equation of one of our curves as $G(x,y)$, and the second as $R(x,y)$, we may write the curve as $G(x,y)=R(x,y)$. If $G(x,y)=R(x,y)$ is the equation of a curve with $a=b=n$, then the equation of the curve with $a=b=n+1$ is $(x-y-2n)(x-y-2n-1)G(x,y)=(x-n)(y+n+1)R(x,y)$; in other words, when $n$ increases by 1, the equation is ``multiplied" by a shifted version of $(x-y)(x-y-1)=x(y+1)$. If this kind of ``multiplication" of two nonsingular curves could be shown to preserve non-singularity, we would have an induction argument to extend the Theorem.\par 
\begin{center} * \hspace{10mm} * \hspace{10mm} * \end{center} \noindent The conclusions we have reached here are significant in their own right: they are, to our knowledge, the first fundamental results established concerning equations of the general form ${x \choose y}={{x-a} \choose {y+b}}$, where $x$ and $y$ vary. This is fundamentally different than considering ${n \choose k}={s \choose r}$ and allowing, say, $k$ and $r$ to vary (an area in which some progress in bounding and tabulating solutions has already been made [2, referenced in 10] [10] [17] [4]).\par
However, it is hardly debatable that the most ambitious and important goal in the examination of repeated coefficients is the proof of Singmaster's [15] conjecture. So far, the most pointed attacks on the conjecture have resulted only in an increasingly tight series of logarithmic bounds\textemdash an approach which \emph{a priori} seems unlikely to yield the desired $O(1)$ result. Kane [8] has stated that his method\textemdash which initially yielded $O(\frac{\log{t}\log{\log{\log{t}}}}{(\log{\log{t}})^2})$, and then was improved by a factor of $\log{\log{t}}$ [9]\textemdash probably cannot be further extended. It may be more fruitful, therefore, to cease considering $N(t)$ as a function to be bounded, and instead only to try analyze when particularly high multiplicities of $t$ occur. We have not done this; our concern has simply been with a certain type of nontrivial repetition. However, a large part of the value of our exploration lies in the fact that the algebraic curves we have used would seem to provide a good basis for pursuing this.\par 
Notice, for instance, that a coefficient occurring six times simply corresponds to an integral intersection between two of our curves beyond a certain $x$-value (the degree of the higher degree curve). A multiplicity of eight corresponds to three curves intersecting at the same point, and so on. Furthermore, any set of the curves has at least some easily calculable number of these common intersections, because of the trivial integral points near the origin which they all share. These correspond to repetitions in the negative triangle. Each large-multiplicity integral intersection between these curves also corresponds to a large integral point on a curve of much lower degree; specifically, if $m$ curves with highest degree $n$ intersect at $(a,b)$, there is an integral point on a curve of degree at most $\frac{m}{n}$ with greater $x$ and $y$ coordinates than $(a-n,b)$.\par
The Singmaster [15] conjecture would be proved by bounding the number of these curves which can share a common intersection beyond a given $x$-value (although this statement is stronger than is necessary; the conjecture only considers integral intersections).\par
The na\"{i}ve way to do this would be to take a general set of some number of these curves, shift them left sufficiently far, and try to show via Nullstellensatz manipulations [13, p. 22] (generating other polynomials in the same ideal) that there could not be a common intersection in the first quadrant. The difficulty, of course, lies in working generally with curves of arbitrary complexity. It should be noted, however, that because of the fact that the Nullstellensatz deals with all intersections, not just integral ones, this strategy is not equivalent to simply manipulating general binomial coefficients themselves. Even if the task is still seemingly difficult, we are able to utilize a more powerful tool.\par
\bigskip \bigskip \bigskip \bigskip \bigskip \bigskip \bigskip \bigskip \bigskip \bigskip \bigskip \bigskip \bigskip \bigskip \bigskip \bigskip \bigskip \bigskip \bigskip \bigskip \bigskip \bigskip \bigskip \bigskip \bigskip \bigskip  \bigskip \bigskip \bigskip \bigskip \bigskip \bigskip \bigskip \bigskip \bigskip \bigskip \bigskip \bigskip \bigskip  
\begin{center} \begin{minipage}{.8\textwidth} {\small \emph{Figs. 1-3. Several of the curves we have considered. The first nontrivial intersection occurs between $a=104$, $b=1$, and $a=110$, $b=2$. It corresponds to ${120 \choose 1}={16 \choose 2}={10 \choose 3}$.}} \end{minipage} \end{center} 
\begin{center}
\includegraphics[width={4.3025in}]{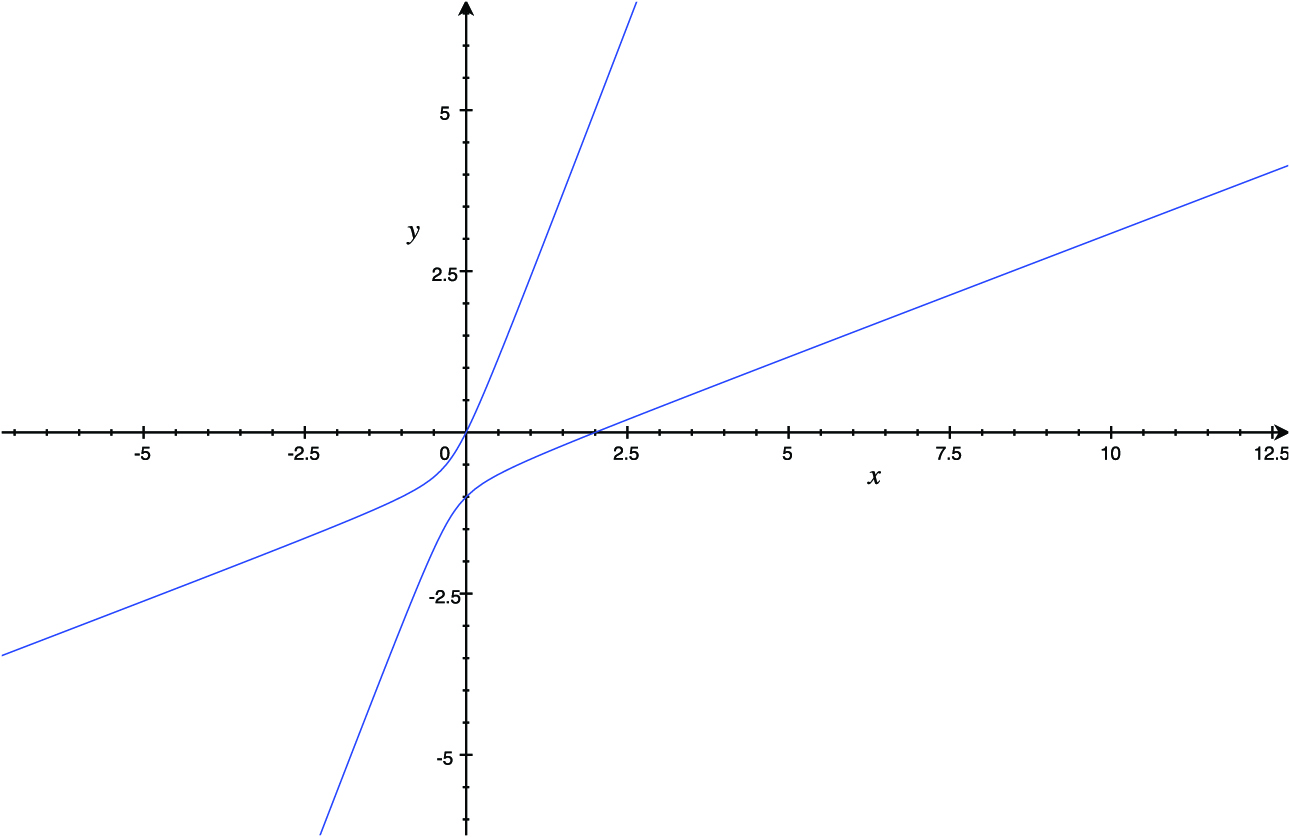}\\
{\small \emph{Fig. 1. Singmaster's curve: $a=1$, $b=1$\\ \hspace{0mm}\\ \hspace{0mm}}}\\
\bigskip
\includegraphics[width={4.3025in}]{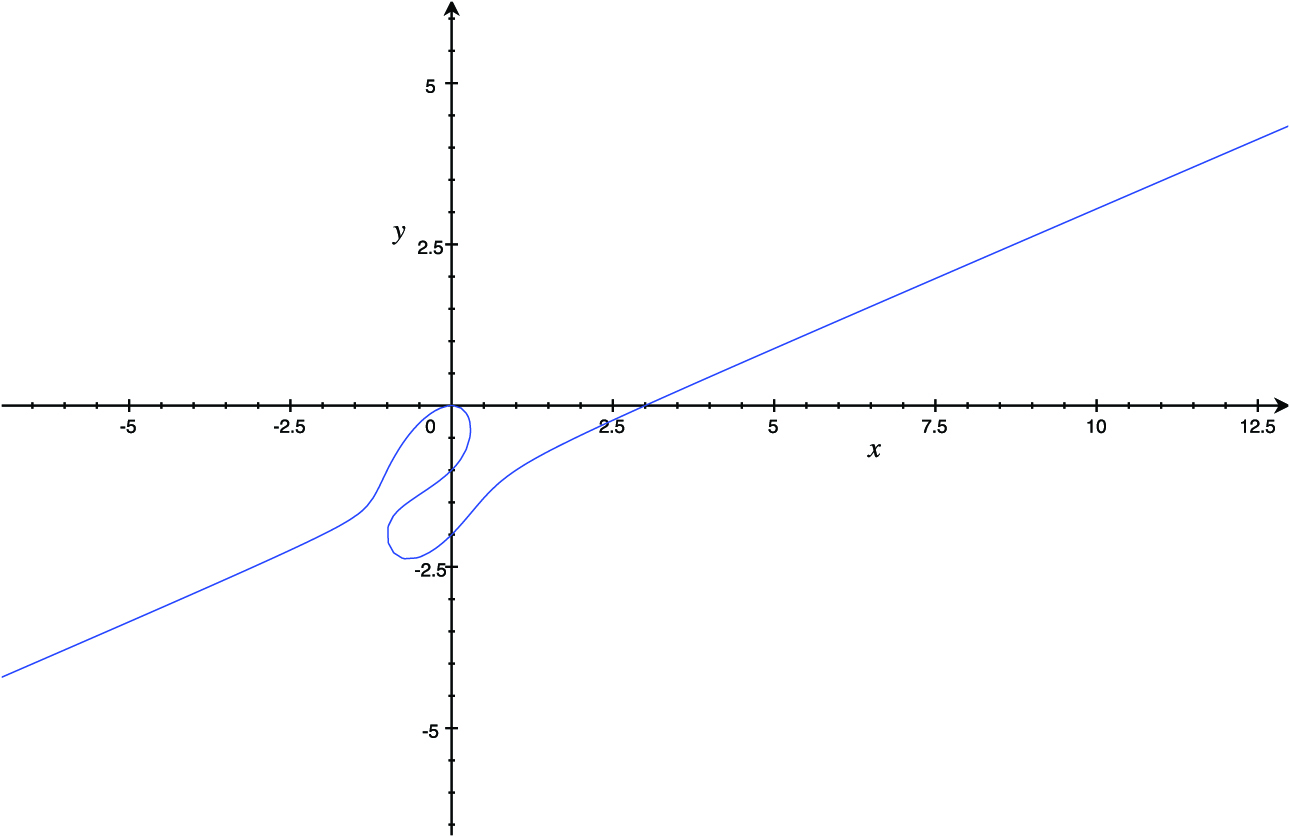}\\
{\small \emph{Fig. 2. $a=1$, $b=2$}}\\ \bigskip
\includegraphics[width={4.3025in}]{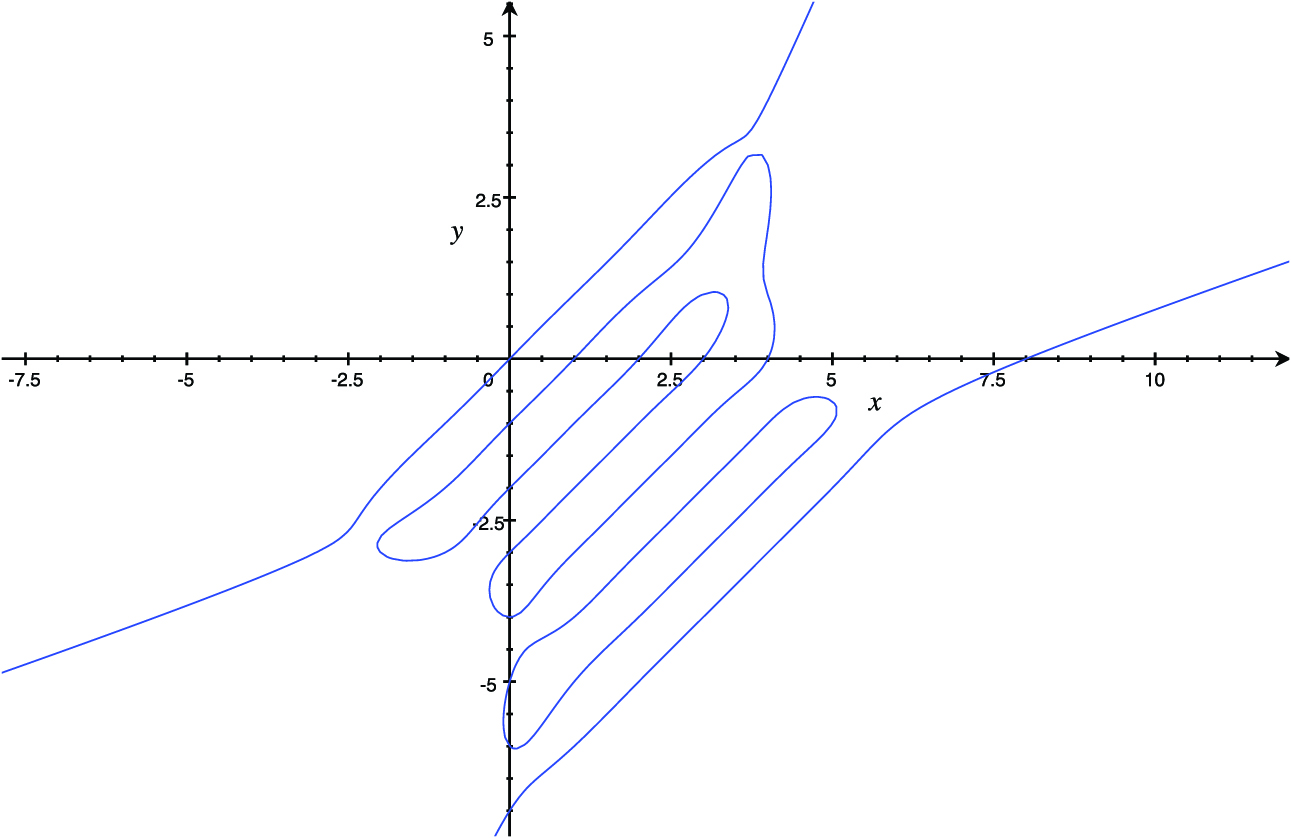}\\
{\small \emph{Fig. 3. $a=5$, $b=3$}} \end{center}
Advanced tools of algebraic geometry are also possibly applicable to this reformulation of the conjecture, although a major strengthening of current knowledge would certainly be necessary first. If a general effective form of Siegel's theorem [7, p. 353] were known, it would be possible to bound the ``height" of integral points on these curves (their coordinate size, essentially). However, the currently known effective methods for genus 1 curves, such as Baker's [3, p. 45] method, generate bounds too large (triple exponential) to be useful, even if they were generalized. More desirable (and more difficult) would be an effective Schmidt subspace theorem, as this would result in an effective form of a corollary [14, p. 5] on simultaneous approximation of algebraic numbers:\par
\begin{center}
\begin{minipage}{.8\textwidth}
Let $\alpha_1$, \dots $\alpha_n$ be algebraic numbers such that $1$, $\alpha_1$, \dots $\alpha_n$ are linearly independent over the rationals. Then for any $\epsilon>0$ there are only finitely many integers $p_1$, \dots $p_n$, $q$ with $q>0$ such that 
\begin{equation*}
\lvert \alpha_1-\frac{p_1}{q} \rvert<q^{-1-1/n-\epsilon}, \dots \lvert \alpha_n-\frac{p_n}{q} \rvert<q^{-1-1/n-\epsilon}.
\end{equation*}
\end{minipage} \end{center} \par \hspace{0mm} \\
\indent If we could find the ratios $\frac{p_i}{q}$ where the various algebraic numbers $\zeta$ associated with a set of our curves are simultaneously approximated, we could find the intersection point. Unfortunately, we have not yet provided a requirement that the approximations to $\zeta$ be as close as is dictated in the corollary.\par
The Singmaster conjecture remains as Paul Erd\H{o}s once described it [cited in 15, p. 385]: a ``very hard" problem. The intent here has been only to introduce a novel form for viewing repeated binomial coefficient problems to which the well-developed tools of 20th century mathematics are at least somewhat applicable. Whether this method can yield a truly new understanding of such an antique, basic, elementary part of mathematics, remains to be seen.
\bigskip
\section*{References}

\singlespacing \noindent[1] Abbott, H. L.; Erd\H{o}s, P.; Hanson, D. \emph{On the number of times an integer occurs as a binomial coefficient}. Amer. Math. Monthly {\bf 81} (1974): 256-261. MR0335283 (49 \#65)\\

\noindent[2] Avanesov, \`{E}. T. \emph{Solution of a problem on figurate numbers}. (Russian) Acta Arithm. {\bf 12} (1966/1967): 409-420. �MR0215784 (35 \#6619)\\

\noindent[3] Baker, A. \emph{Transcendental number theory}. Cambridge University Press, London-New York, 1975.� MR0422171 (54 \#10163)\\

\noindent[4] Bugeaud, Y.; Mignotte, M.; Siksek, S.; Stoll, M.; Tengely, Sz. \emph{Integral points on hyperelliptic curves}. Algebra \& Number Theory {\bf 2} (2008), no. 8: 859-885.� arXiv:0801.4459\\

\noindent[5] Coolidge, J. L. \emph{A treatise on algebraic plane curves}. Dover Publications, New York, 1959 [1931]. MR0120551 (22 \#11302)\\

\noindent[6] Hardy, G. H.; Wright, E. M. \emph{An introduction to the theory of numbers}. 3rd. edn., Oxford University Press, New York, 1954 [1938]. MR0067125\\

\noindent[7] Hindry, M.; Silverman, J. H. \emph{Diophantine geometry: an introduction}. Graduate Texts in Mathematics, {\bf 201}, Springer, New York, 2000.� MR1745599 (2001e:11058)\\

\noindent[8] Kane, D. M. \emph{New bounds on the number of representations of t as a binomial coefficient}. Integers: Electronic  J. of Combinatorial Number Theory {\bf 4} (2004), \#A07: 1-10. �MR2056013\\

\noindent[9] Kane, D. M. \emph{Improved bounds on the number of ways of expressing t as a binomial coefficient}. Integers: Electronic� J. of Combinatorial Number Theory {\bf 7} (2007), \#A53: 1-7.� MR2373115  (2008m:05014)\\

\noindent[10] Kiss, P. \emph{On the number of solutions of the Diophantine equation $\backslash$binom(x, p) = $\backslash$binom(y, 2)}. Fib. Quart. {\bf 26} (1988), no. 2: 127-129. �MR0938585 �(89f:11050)\\

\noindent[11] Lind, D. A. \emph{The quadratic field Q($\sqrt{5}$) and a certain Diophantine equation}. Fib. Quart. {\bf 6} (1968), no. 3: 86-93.� MR0231784 �(38 \#112)\\

\noindent[12] Nagell, T. \emph{Introduction to number theory}. 2nd edn., Chelsea Publishing Company, New York, 1964 [1951].\\

\noindent[13] Prasolov, V. V. \emph{Polynomials}, trans. D. Leites. Algorithms and Computation in Mathematics, {\bf 11}, Springer, Berlin, 2004. �MR2082772 (2005f:12001)\\

\noindent[14] Schlickewei, H. P. \emph{The mathematical work of Wolfgang Schmidt}. In Schlickewei, H. P.; Tichy, R. F.; Schmidt, K. D., eds. \emph{Diophantine approximation: festschrift for Wolfgang Schmidt}. Springer, Vienna-New York, 2000: 1-14.\\

\noindent[15] Singmaster, D. \emph{How often does an integer occur as a binomial coefficient?} Amer. Math. Monthly {\bf 78} (1971), no. 4: 385-386.� MR1536288\\

\noindent[16] Singmaster, D. \emph{Repeated binomial coefficients and Fibonacci numbers}. Fib. Quart. {\bf 13} (1975), no. 4: 295-298. MR0412095 (54 \#224)\\

\noindent[17] de Weger, B. M. M. \emph{Equal binomial coefficients: some elementary considerations}. J. Number Theory {\bf 63} (1997), no. 2: 373-386.� MR1443768 (98b:11027)

 \end{document}